\documentclass[11pt,twoside,a4paper]{article}
\usepackage{amsmath,amssymb,amsthm,ifpdf}

\usepackage{graphicx,psfrag}

\newtheorem{theorem}{Theorem}[section]
\newtheorem{lemma}[theorem]{Lemma}
\newtheorem{prop}[theorem]{Proposition}
\newtheorem{cor}[theorem]{Corollary}

\newcommand{\CAT}{\operatorname{CAT}}

\theoremstyle{definition}

\theoremstyle{remark}
\newtheorem{remark}[theorem]{Remark}
\newtheorem{definition}[theorem]{Definition}

\newcommand{\skap}{\operatorname{sn}_{\kappa}}

\begin{document}

\title{Jung's Theorem and fixed points for $p$-uniformly convex spaces}
\author{Renlong Miao}

\maketitle

\section{Introduction}

\begin{definition} \cite{bcl}\cite{kk}
A metric space $(X,d)$ is called a $p$-uniformly convex with parameter $k>0$, if $(X, d)$ is a geodesic space and for any
three points $x, y, z \in X$ , any minimal geodesic $\gamma := (\gamma_t )_{t\in[0,1]}$ in $X$ with $\gamma_0 = x, \gamma_1 = y$, and
the midpoint $m$ of $x$ and $y$,
\[
  d^p(z,m)\leq \frac{1}{2}d^p(z,x)+\frac{1}{2}d^p(z,y)-\frac{k}{8}d^p(x,y).
\]
\end{definition}
By definition, putting $z = \gamma_0$ or $z=m$, we see $k \in (0, 2]$ and $p \in [1, \infty)$. The inequality yields the strict convexity of $Y \ni x \to d^p (z, x)$ for a fixed $z \in Y$ . Any closed convex subset of a $p$-uniformly convex space is again a $p$-uniformly convex space with the same
parameter. Any $L^p$ space over a measurable space is $p$-uniformly convex with parameter
$k = 2^{3-p}$ provided $p > 2$, and it is 2-uniformly convex with parameter $k = 2(p - 1)$ provided $1 < p \leq 2$. More details are provided in \cite{bcl}. A geodesic space is $\CAT(0)$ space if and only if it is a $2$-uniformly convex space with parameter
$k =2$. Ohta \cite{oo} proved that for $\kappa > 0$ any $\CAT(\kappa)$-space $Y$ with $diam(Y ) < R_\kappa /2$ is a 2-uniformly convex space with parameter $\{(\pi - 2 \sqrt{\kappa}\varepsilon
) \tan \sqrt{\kappa}\varepsilon\}$
for any $\varepsilon \in (0, R_\kappa /2 - diam(Y)]$.

The classical Jung theorem gives an optimal upper estimate for the radius of a bounded subset of $\mathbb{R}^n$ in terms of its diameter and dimention. In \cite{ls}, Lang and Schroeder also proved the similar Jung's theorem for $\CAT(\kappa)$ spaces. Here we will give an upper bound for the radius of a bounded subset of $p$-uniformly convex spaces.
\begin{theorem}\label{jung}
Let $X$ be a complete $p$-uniformly convex space and $S$ be a nonempty bounded subset of $X$. Then there exists a unique closed circumball $B(z, rad(S))$ of $S$ and 
\[
  rad(S)\leq (1+\frac{2^{p-3}k}{2^{p-1}-1})^{-\frac{1}{p}}diam(S).
\]
\end{theorem}
\begin{remark}
For $p=2,k=2$, our result coincides with the classical Jung theorem for $\CAT(0)$ spaces. Using a similar method, we can give a shorter proof of the Jung theorem of $\CAT(\kappa)$ spaces \cite{ls}.
\end{remark}

\begin{theorem}\cite{ls}
Let $X$ be a complete $\CAT(\kappa)$ space and $S$ a nonempty bounded subset of $X$. In case $\kappa>0$ assume that $rad(S)<\pi/(2\sqrt{\kappa})$. Then there exists a unique closed circumball $B(z, rad(S))$ of $S$ and 
\[
  \skap rad(S)\leq \sqrt{2} \skap (diam(S)/2),
\]
where $\skap$ is the function
\begin{equation*}
  \skap(x)=
     \begin{cases}
     \frac{1}{\sqrt{\kappa}}\sin(\sqrt{\kappa}x) & \text{if} \,\,\,\kappa>0, \\
     x & \text{if} \,\,\,\kappa=0, \\
     \frac{1}{\sqrt{-\kappa}}\sinh(\sqrt{-\kappa}x) & \text{if} \,\,\,\kappa<0.
     \end{cases}
\end{equation*}
\end{theorem}

A mapping $T : M \to M$ of a metric space $(M, d)$ is said to be uniformly $L$-lipschitzian if there
exists a constant $L$ such that $d(T^n x, T^n y)\leq Ld(x,y)$, for all $x, y \in M$ and $n \in \mathbb{N}$.  In \cite{dks} there is a following result for $\CAT(0)$ spaces
\begin{theorem}\cite{dks}
Let $(X, d)$ be a bounded complete $\CAT(0)$ space. Then every uniformly
$L$-lipschitzian mapping $T : X \to X$ with $L < \sqrt{2}$ has a fixed point.
\end{theorem}
\begin{remark}
In \cite{gk}, Baillon gave a uniformly $\frac{\pi}{2}$-lipschitzian mapping of Hilbert spaces which is fixed point free.
\end{remark}

In \cite{l2}, Lim proved a general theorem for $L^p$ spaces
\begin{theorem}\cite{l2}
Let $K$ be a closed convex bounded nonempty subset of $L^p$, $2<p<\infty$, then every uniformly
$L$-lipschitzian mapping $T : K \to K$ with $L <L_0$ has a fixed point. Here
\[
  L_0\geq (1+\frac{1}{2^{p-1}}p^{\frac{p}{2}}(p-1)^{1-p}(p-2)^{\frac{p}{2}-1})^{\frac{1}{p}}>(1+\frac{1}{2^{p-1}})^{\frac{1}{p}}.
\] 
\end{theorem}

We prove similar results for $p$-uniformly convex spaces,
\begin{theorem}\label{main}
Let $(X,d)$ be a bounded complete $p$-uniformly convex space with parameter $k>0$. Then there exists a constant $C=(1+\frac{2^{p-3}k}{2^{p-1}-1})^{\frac{1}{p}}$ such that for every uniformly $L$-lipschitzian mapping $T: X \to X$ with $L<C$ has a fixed point.
\end{theorem}

\begin{remark}
For $\CAT(0)$ spaces we have $p=2, k=2$, hence the Lifschitz constant $L(X)\geq \sqrt{2}$ which is coincide with the result in \cite{dks}. For $L^p$ spaces we have $k=\frac{1}{2^{p-3}}$, hence $L(X) \geq (1+\frac{1}{2^{p-1}-1})^{\frac{1}{p}}$.
\end{remark}

This paper is organized as follows. In Section 2 we introduce the classical Jung theorem and prove a similar one for $p$-uniformly convex spaces. Moreover using the same method, we can give a shorter proof for $\CAT(\kappa)$ spaces. In Section 3 we show a general fixed point theorem for $p$-uniformly convex spaces which generalize the results in \cite{ef}\cite{kp}. In Section 4 we prove that $p$-uniformly convex spaces enjoy the Property (P) which is defined by Lim and Xu. In Section 5, we generalize the result about $\Delta$-convergence from \cite{ef}\cite{kp} for $\CAT(\kappa)$ spaces.

\section{Jung's Theorem for $p$-uniformly convex space}\label{sec:jung theorem}
Let $(X,d)$ be a metric space. For a nonempty bounded subset $D\subset X$, set
\[
  r_x(D)=\sup \{d(x,y):y \in D\}, x\in X;
\]
\[
  rad(D)=\inf \{r_x(D):x\in X\};
\]
\[
 diam (D)=\sup \{d(x,y): x,y \in D\}.
\]
Clearly $rad(D)\leq diam(D)\leq 2rad(D)$. Jung's theorem states that each bounded subset $D$ of $\mathbb{R}^n$ is contained in a unique closed ball with $rad(D)$, where
\[
  rad(D) \leq \sqrt{\frac{n}{2(n+1)}}diam(D).
\]

\begin{theorem}\label{jung's theorem}
Let $X$ be a complete geodesic $p$-uniformly convex space and $S$ be a nonempty bounded subset of $X$. Then there exists a unique closed circumball $B(z, rad(S))$ of $S$ and
\[
  rad(S)\leq (1+\frac{2^{p-3}k}{2^{p-1}-1})^{-\frac{1}{p}}diam(S).
\]
\end{theorem}
\begin{proof}
For any bounded closed subset $S\subset X$, choose $\{x_n\} \in S $ such that $\max\{\limsup_{n\to \infty}d(x,x_n),x\in S\}= rad(S)$. Now we want to show that $\{x_n\}$ is a Cauchy sequence. Suppose not, then there exists $\varepsilon>0$ such that for any $N\in \mathbb{N}$ there exist $i,j \geq N$ such that $d(x_i,x_j)\geq \epsilon$. Choose $m_i$ as the midpoint of the segment $[x_ix_j]$ and according to the $p$-uniformly convexity, we have
\begin{eqnarray*}
  d^p(y,m_i)&\leq& \frac{1}{2}d^p(y,x_i)+\frac{1}{2}d^p(y,x_j)-\frac{k}{8}d^p(x_i,x_j)
\end{eqnarray*}
for $\forall y\in S$. Choose $N \in \mathbb{N}$ large enough such that $d^p(y,x_i)<rad^p(S)+\frac{k}{16}\varepsilon^p$ for all $i\geq N$ and $\forall y \in S$. Then we have
\[
  d^p(y,m_i)\leq rad^p(S)-\frac{k}{16}\varepsilon^p
\]
which means $d(y,m_i)<rad(S)$ for all $i\geq N$ and $\forall y \in S$. Contradicts with the definition of $rad(S)$.\\
Denote $z$ as the circumcenter of $S$ and choose $w \in S$ such that $d(z,w)=rad(S)$. Choose $z_i \in S$ as the the midpoint of segment $[z,z_{i-1}]$, where $z_0=w$ and $w_i \in S$ such that $d(z_i,w_i) \geq rad(S)$. Applying to the $p$-uniformly convexity we have
\[
  d^p(z_i, w_i)\leq (\frac{2^{i+1}-1}{2^{i+1}}l^p+\frac{1}{2^{i+1}})diam^p(S)-\frac{k}{2^{i+3}}\sum_{j=0}^i\frac{1}{2^{j(p-1)}}l^pdiam^p(S).
\]
i.e.
\[
  l^p diam^p(S)\leq (\frac{2^{i+1}-1}{2^{i+1}}l^p+\frac{1}{2^{i+1}})diam^p(S)-\frac{k}{2^{i+3}}\sum_{j=0}^i\frac{1}{2^{j(p-1)}}l^pdiam^p(S)
\]
where $l=\frac{rad(S)}{diam(S)}$.
i.e.
\[
  l^p(1+\frac{k}{4}\sum_{j=0}^i\frac{1}{2^{j(p-1)}}) \leq 1.
\]
Let $i \to \infty$, we obtain
\[
  l^p\leq (1+\frac{k}{4}\frac{2^{p-1}}{2^{p-1}-1})^{-1}=(1+\frac{2^{p-3}k}{2^{p-1}-1})^{-1}.
\]
Hence we obtain
\[
  rad(S)\leq (1+\frac{2^{p-3}k}{2^{p-1}-1})^{-\frac{1}{p}}diam(S).
\]
\end{proof}

Now we give a shorter proof of Jung's theorem for $\CAT(\kappa)$ spaces.
\begin{theorem}\cite{ls}\label{LS}
Let $X$ be a complete $\CAT(\kappa)$ space and $S$ a nonempty bounded subset of $X$. In case $\kappa>0$ assume that $rad(S)<\pi/(2\sqrt{\kappa})$. Then there exists a unique closed circumball $B(z, rad(S))$ of $S$ and 
\[
  \skap rad(S)\leq \sqrt{2} \skap (diam(S)/2),
\]
\end{theorem}

Here we give the prove for $\kappa=1, -1$.
\begin{lemma}\label{cat(1)}
Let $X$ be a $\CAT(1)$ space and $x, y, z \in X$ such that $d(x, y) + d(y, z ) + d(z , x) < 2\pi$. Let $t \in [0, 1]$ and $u$ is on the segment $[y,z]$ such that $d(y,u)=td(y,z)$. Then
\[
  \cos d(x, u) \sin d(y, z) \geq \cos d(x, y) \sin (td(y, z )) + \cos d(x, z) \sin ((1-t )d(y, z)).
\]
\end{lemma}

\begin{lemma}\label{cat(-1)}
Let $X$ be a $\CAT(-1)$ space and $x, y, z \in X$. Let $t \in [0, 1]$ and $u$ is on the segment $[y,z]$ such that $d(y,u)=td(y,z)$. Then
\begin{eqnarray*}
  \cosh d(x, u) \sinh d(y, z) &\leq& \cosh d(x, y) \sinh (td(y, z )) \\
&&+ \cosh d(x, z) \sinh ((1-t )d(y, z)).
\end{eqnarray*}
\end{lemma}
\begin{proof}
Consider a comparison triangle $\triangle(x,y,z)$ and apply the $\CAT(\kappa)$ inequality, we got the results.
\end{proof}

Now we prove the Theorem \ref{LS}
\begin{proof}
The uniqueness and existence of the circumball are directly from the result of $p$-uniformly convexity.
Case $\kappa=1$: for any bounded closed subset $S\subset X$, denote $z$ as the circumcenter of $S$ and choose $w \in S$ such that $d(z,w)=rad(S)$. $u_t$ is on the segment $[z,w]$ such that $d(z,u_t)=td(z,w)$ and $w_t \in S$ such that $d(u_t,w_t) \geq rad(S)$. According to the Lemma \ref{cat(1)} we have
\[
  \cos d(w_t, u_t) \sin d(z, w) \geq \cos d(w_t, w) \sin (td(z, w )) + \cos d(w_t, z) \sin ((1-t )d(z, w))
\]
i.e.
\[
  \cos rad(S) \sin rad(S) \geq \cos diam(S) \sin (t rad(S)) + \cos rad(S) \sin ((1-t )rad(S))
\]
\[
  2\cos rad(S) \sin \frac{t}{2}rad(S) \cos (1-\frac{t}{2}) rad(S)  \geq \cos diam(S) \sin (t rad(S)).
\]
Hence $\frac{\cos rad(S) \cos (1-\frac{t}{2}) rad(S)}{ \cos \frac{t rad(S)}{2}}\geq \cos diam(S)$ for all $t \in (0,1)$. Let $t \to 0$, we obtain
\[
  \cos^2 rad(S) \geq \cos diam(S)
\]
i.e.
\[
  1-\cos^2 rad(S)=\sin^2 rad(S) \leq 1-\cos diam(S)=2\sin^2 diam(S)/2.
\]
Thus $\sin rad(S)\leq \sqrt{2}\sin diam(S)/2$.\\
Case $\kappa=-1$: it is similar as the case of $\kappa=1$.
\end{proof}

\section{Fixed points in $p$-uniformly convex space}
We now turn to the definition of the Lifsic character of a metric space $X$. Balls
in $X$ are said to be $c$-regular if the following holds: For each $k < c$ there exist
$\mu, \alpha \in (0, 1)$ such that for each $x, y \in X$ and $r > 0$ with $d(x, y) \geq (1- \mu)r$, there
exists $z \in X$ such that
\[
  B (x; (1 + \mu) r)\cap B (y; k (1 + \mu) r) \subset B (z; \alpha r)
\]
The Lifshitz character $L(X)$ of $X$ is defined as follows:
\[
  L(X) = \sup \{c \geq 1 : \text{balls in X are c-regular} \}.
\]

\begin{theorem}\cite{ll}
Let $(X, d)$ be a bounded complete metric space. Then every uniformly $L$-lipschitzian mapping $T : X \to X$ with $L < L(X)$ has a
fixed point.
\end{theorem}

\begin{theorem}\label{main}
Let $(X,d)$ be a bounded complete $p$-uniformly convex space with parameter $k>0$. Then there exists a constant $C=(1+\frac{2^{p-3}k}{2^{p-1}-1})^{\frac{1}{p}}$ such that for every uniformly $L$-lipschitzian mapping $T: X \to X$ with $L<C$ has a fixed point.
\end{theorem}
\begin{proof}
We just have to show the Lifshitz character of $X$
\[
  L(X)\geq (1+\frac{2^{p-3}k}{2^{p-1}-1})^{\frac{1}{p}}.
\]
For each $x, y \in X$ and $r > 0$ with $d(x, y) \geq (1- \mu)r$, denote
\[
  A:=B (x; (1 + \mu) r)\cap B (y; l(1 + \mu) r).
\]
Choose the midpoint $m_0$ between $x$ and $y$, for any $z\in A$ applying the $p$-uniformly convexity, we have
\[
  d(z, m_0)^p\leq \frac{1}{2}(1+\mu)^p r^p+\frac{1}{2}l^p(1+\mu)^p r^p-\frac{k}{8}(1-\mu)^p r^p.
\]
Let $\mu$ small enough such that $(1+\mu)^p<1+\varepsilon$. Since $(1+\mu)^p+(1-\mu)^p \geq 2$, we obtain
$(1-\mu)^p \geq 1-\varepsilon$. Hence
\[
  d(z, m_0)^p\leq (\frac{1}{2}+\frac{1}{2}l^p-\frac{k}{8})r^p+M\varepsilon r^p.
\]
Choose $m_1$ be the midpoint between $x$ and $m_0$, for any $z\in A$ applying the $p$-uniformly convexity again, we have
\[
  d(z, m_1)^p\leq \frac{1}{2}(1+\mu)^p r^p+\frac{1}{2}d(z, m_0)^p-\frac{k}{8}(1-\mu)^p (\frac{r}{2})^p.
\]
i.e.
\begin{eqnarray*}
  d(z, m_1)^p&\leq& (\frac{1}{2}+\frac{1}{2}(\frac{1}{2}+\frac{1}{2}l^p-\frac{k}{8}))r^p-\frac{k}{8}(\frac{r}{2})^p+(\frac{1}{2}+\frac{1}{2^p})M\varepsilon r^p\\
&\leq& (\frac{3}{4}+\frac{1}{4}l^p)r^p-\frac{k}{8}(\frac{1}{2}+\frac{1}{2^p})r^p+M\varepsilon r^p.
\end{eqnarray*}
Inductivly, choose $m_i$ as the midpoint of $x$ and $m_{i-1}$. Therefore we have
\[
  d(z, m_i)^p\leq (\frac{2^{i+1}-1}{2^{i+1}}+\frac{1}{2^{i+1}}l^p)r^p-\frac{k}{2^{i+3}}\sum_{j=0}^i\frac{1}{2^{j(p-1)}}r^p+\frac{1}{2^{i-2}}M\varepsilon r^p.
\]
Let $\alpha \to 1, \mu\to 0$, we get
\[
  (\frac{2^{i+1}-1}{2^{i+1}}+\frac{1}{2^{i+1}}l^p)r^p-\frac{k}{2^{i+3}}\sum_{j=0}^i\frac{1}{2^{j(p-1)}}r^p\leq r^p.
\]
i.e.
\[
  l^p\leq 1+\frac{k}{4}\sum_{j=0}^i\frac{1}{2^{j(p-1)}}.
\]
Let $i \to \infty$, we obtain
\[
  l^p\leq 1+\frac{k}{4}\frac{2^{p-1}}{2^{p-1}-1}=1+\frac{2^{p-3}k}{2^{p-1}-1}.
\]
Hence $L(X)\geq (1+\frac{2^{p-3}k}{2^{p-1}-1})^{\frac{1}{p}}$.
\end{proof}

\section{$p$-uniformly convex spaces and Property (P)}
A subset $A$ of $X$ is said to be admissible if $cov(A)=A$ here
\[
cov(A)= \cap \{B: B \text{ is a closed ball and } A \subset B\}.
\]
The number
\[
  \tilde{N}(X):= \sup \{\frac{rad(A)}{diam(A)} \},
\]
where the supremum is taken over all nonempty bounded admissible subsets $A$ of $X$ for which $\delta(A)>0$, is called the normal structure coefficient of $X$. If $\tilde{N}(X)\leq c$ for some constant $c<1$ then $X$ is said to have uniform normal structure.

Lim and Xu introduced the so-called property (P) for metric spaces. A metric space $(X,d)$ is said to have property (P) if given two bounded sequences $\{x_n\}$ and $\{z_n\}$ in $X$, there exists $z\in \cap_{n\geq 1}cov(\{ z_j: j \geq n\})$ such that
\[
  \limsup_n d(z,x_n) \leq \limsup_j \limsup_n d(z_j,x_n).
\]

The following theorem is the main result of \cite{th}
\begin{theorem}\cite{th}
Let $(X,d)$ be a complete bounded metric space with both property (P) and uniform normal structure. Then every uniformly $L$-lipschitzian mapping $T: X\to X$ with $L<\tilde{N}(X)^{-\frac{1}{2}}$ has a fixed point.
\end{theorem}
From Theorem \ref{jung's theorem}, for any $p$-uniformly convex space $X$ we have $\tilde{N}(X)\leq (1+\frac{2^{p-3}k}{2^{p-1}-1})^{-\frac{1}{p}}<1$. Hence $X$ has uniform normal structure. In this section we show that every complete geodesic $p$-uniformly convex spaces have property (P). 

Let $\{x_n\}$ be abounded sequence in a complete geodesic $p$-uniformly convex space and let $K$ be a closed and convex subset of $X$. Define $\phi: X\to \mathbb{R}$ by setting $\phi(x)= \limsup_{n\to \infty}d(x,x_n)$, $x\in X$.

\begin{prop}\label{exist}
There exists a unique point $u\in K$ such that
\[
  \phi(u)=\inf_{x\in K} \phi(x)
\]
\end{prop}
\begin{proof}
Let $r=\inf_{x\in K} \phi(x)$ and let $\epsilon >0$. Then by assumption there exists $x\in K$ such that $\phi(x)<r+\epsilon$; thus for $n$ sufficiently large $d(x,x_n)<r+\epsilon$, i.e., for $n$ sufficiently large $x\in B(x_n,r+\epsilon)$. Thus
\[
  C_{\epsilon}:= \cup_{k=1}^{\infty}(\cap_{i=k}^{\infty}B(x_i,r+\epsilon)\cap K)\neq \emptyset.
\]
As the ascending union of convex sets, clearly $C_{\epsilon}$ is convex. Also the closure $\bar{C_{\epsilon}}$ is also convex. Therefore
\[
  C:=\cap_{\epsilon>0} \bar{C_{\epsilon}}\neq \emptyset.
\]
Clearly for $u \in C, \phi \leq r$. Uniqueness of such a $u$ follows from the $p$-uniformly convexity.
\end{proof}

In the view of the above, $X$ has property (P) if given two bounded sequences $\{x_n\}$ and $\{z_n\}$ in $X$, there exists $z\in \cap_{n=1}^{\infty} cov \{z_j: j\geq n\}$ such that
\[
  \phi(z) \leq \limsup_{j\to \infty} \phi(z_j),
\]
where $\phi$ is defined as above.

\begin{prop}
A complete geodesic $p$-uniformly convexity has property (P).
\end{prop}
\begin{proof}
Let $\{x_n\}$ and $\{z_n\}$ be two bounded sequences in $X$ and define $\phi: X\to \mathbb{R}$ by setting $\phi(x)= \limsup_{n\to \infty}d(x,x_n)$, $x\in X$. For each $n$, let
\[
  C_n:=cov \{z_j: j\geq n\}.
\]
By Proposition \ref{exist} there exists a unique point $u_n\in C_n$ such that
\[
  \phi(u_n)=\inf_{x\in C_n} \phi(x).
\]
Moreover, since $z_j\in C_n$ for $j \geq n$, $\phi(u_n) \leq \phi(z_j)$ for all $j \geq n$. Thus $\phi(u_n) \leq \limsup_{j\to \infty} \phi(z_j)$ for all $n$. We assert that $\{u_n\}$ is a Cauchy sequence.

To see this, suppose not. Then there exists $\epsilon>0$ such that for any $N\in \mathbb{N}$ there exist $i,j \geq N$ such that $d(u_i,u_j)\geq \epsilon$. Also, since the sets $\{C_n\}$ are descending, the sequence $\{\phi(u_n)\}$ is increasing. Let $d:=\lim_{n \to \infty} \phi(u_n)\geq \frac{\epsilon}{2}$. Choose $\xi>0$ so small that $\xi <(\cosh \frac{\epsilon}{4}-1)\sinh \frac{d}{4}$, and choose $N$ so large that $\sinh \frac{d}{2}<\sinh \phi(u_j)-\xi\leq \sinh \phi(u_i)\leq \sinh \phi(u_j)\leq \sinh d$ if $j\geq i \geq N$. Let $m_j$ denote the midpoint of the geodesic joining $u_i$ and $u_j$, and let $n \in \mathbb{N}$, Then by the $p$-uniformly convexity
\begin{eqnarray*}
  d^p(m_j,x_n)\leq \frac{1}{2}d^p(u_i,x_n)+\frac{1}{2}d^p(u_j,x_n)-\frac{k}{8}d^p(u_i,u_j).
\end{eqnarray*}
This implies
\[
\phi^p(m_j)\leq \phi^p(u_j)-\frac{k}{8}\epsilon^p.
\]
Since $m_j\in C_j$, this contradicts the definition of $u_j$.

This proves that $\{u_n\}$ is a Cauchy sequence. Consequently there exists a $z\in \cap_{n=1}^{\infty}C_n$ such that $\lim_{n\to \infty}u_n=z$ and, since $\phi$ is continuous, $\lim_{n\to \infty}\phi(u_n)=\phi(z)$. Hence we conclude that
\[
  \phi(z)\leq \limsup_{j\to \infty} \phi(z_j).
\]
\end{proof}

\section{Basic properties of $\Delta$-convergence}
In this section we show that $\Delta$-convergence can be used in $p$-uniformly convex spaces in a similar way as it
is used in \cite{kp} for $\CAT(0)$ spaces, obtaining a collection of similar results. To show this we
begin with the definition of $\Delta$-convergence.

Let $X$ be a complete $p$-uniformly convex space and $(x_n)$ a bounded sequence in $X$. For $x \in X$ set
\[
  r(x, (x_n)) = \limsup_{n\to \infty} d(x, x_n).
\]
The asymptotic radius $r((x_n))$ of $(x_n)$ is given by
\[
  r((x_n)) = \inf \{r(x, (x_n )) : x \in X\},
\]
the asymptotic radius $r_C((x_n))$ with respect to $C \subset X$ of $(x_n)$ is given by
\[
  r_C ((x_n)) = \inf \{r(x, (x_n)) : x \in C\},
\]
the asymptotic center $A((x_n))$ of $(x_n)$ is given by the set
\[
  A((x_n)) = \{x \in X : r(x, (x_n)) = r((x_n))\},
\]
and the asymptotic center $A_C ((x_n))$ with respect to $C\subset X$ of $(x_n)$ is given by the set
\[
  A_C((x_n)) = \{x \in C : r(x, (x_n)) = r_C ((x_n))\}.
\]
From Proposition \ref{exist}, we have the following
\begin{prop}
Let $X$ be a complete $p$-uniformly convex space, $C \subset X$ nonempty bounded, closed and convex, and
$(x_n)$ a bounded sequence in $X$. Then $A_C ((x_n))$ consists of exactly one point.
\end{prop}

\begin{definition}
A sequence $(x_n)$ in $X$ is said to $\Delta$-converge to $x \in X$ if $x$ is the unique asymptotic
center of $(u_n)$ for every subsequence $(u_n)$ of $(x_n)$. In this case we write $\Delta$-$\lim_{n\to \infty} x_n = x$ and call
$x$ the $\Delta$-limit of $(x_n)$.
\end{definition}
The next result follows as a consequence of the previous proposition.
\begin{cor}
Let $X$ be a complete bounded $p$-uniformly convex space and $(x_n)$ a sequence in $X$.
Then $(x_n)$ has a $\Delta$-convergent subsequence.
\end{cor}

Next we show that we can give analogs in $2$-uniformly convex spaces to those other results in Section 3 of
\cite{kp} for $\CAT(0)$ spaces. Notice that this generalizes these results. In all the next definitions $X$ is a $2$-uniformly convex
space and $K \subset X$ bounded and convex.
\begin{definition}
A mapping $T : K \to X$ is said to be of type $\Gamma$ if there exits a continuous strictly
increasing convex function $\gamma : \mathbb{R}^+ \to \mathbb{R}^+$ with $\gamma(0) = 0$ such that, if $x, y \in K$ and if $m$ and $m'$ are
the mid-points of the segments $[x, y]$ and $[T(x), T(y)]$ respectively, then
\[
  \gamma(d(m , T(m))) \leq |d(x, y) - d(T(x), T(y))|.
\]
\end{definition}

\begin{definition}
A mapping $T : K \to X$ is called $\alpha$-almost convex for $\alpha : \mathbb{R}^+ \to \mathbb{R}^+$ continuous,
strictly increasing, and $\alpha(0) = 0$, if for $x, y \in K$,
\[
  J_T (m) \leq \alpha(\max\{J_T (x), J_T (y)\}),
\]
where $m$ is the mid-point of the segment $[x, y]$, and $J_T(x) := d(x, T(x))$.
\end{definition}

\begin{definition}
A mapping $T : K \to X$ is said to be of convex type on $K$ if for $(x_n),(y_n)$ two
sequences in $K$ and $(m_n)$ the sequence of the mid-points of the segments $[x_n , y_n]$,
\begin{equation*}
\begin{cases}
\lim_{n\to \infty} d(x_n, T(x_n))= 0,\\
\lim_{n\to \infty} d(y_n ,T(y_n)) = 0
\end{cases}
\Rightarrow \lim_{n\to \infty} d(m_n , T(m_n)) = 0.
\end{equation*}
\end{definition}

\begin{prop}
Let $K$ be a nonempty bounded closed convex subset of a $2$-uniformly convex space $X$ and let $T :
K \to X$, then the following implications hold:
\[
  T \;\text{is nonexpansive} \Rightarrow T \;\text{is of type} \;\Gamma \Rightarrow \\
  T \;\text{is} \;\alpha\text{-almost convex} \Rightarrow T \;\text{is of convex type}.
\]
\end{prop}

\begin{lemma}\cite{oo}
Let $X$ be a $2$-uniformly convex geodesic space with some parameter $k>0$.
For any $x,y,z \in X$, denote $m$ as the midpoint of the segments $[y,z]$. Then, we have
\[
  d^2(x,m) \leq \frac{4}{k}\{\frac{1}{2}d^2(x,y)+\frac{1}{2}d^2(x,z)-\frac{1}{4}d^2(y,z)\}
\]
\end{lemma}

Now we prove the above proposition
\begin{proof}
For the first implication, let $m$ denote the midpoint of the segment $[x, y]$ for $x, y \in K$,
and let $m'$ denote the midpoint of the segment $[T (x), T (y)]$. From the lemma, we have
\begin{eqnarray*}
  d^2(m',T(m))&\leq& \frac{4}{k}\{\frac{1}{2}d^2(T(m),T(x))+\frac{1}{2}d^2(T(m),T(y))-\frac{1}{4}d^2(T(x),T(y))\}\\
              &\leq& \frac{1}{k}(d^2(x,y)-d^2(T(x),T(y)))\\
              &\leq& \frac{2 diam(K)}{k}(d(x,y)-d(T(x),T(y)))
\end{eqnarray*}
Thus it suffices to take $\gamma(t)=\frac{k}{2 diam(K)}t^2$ to complete the first implication.

In order to prove the second implication, we have first
\begin{eqnarray*}
  J_T (m)=d(m,T(m))&\leq& d(m,m')+d(m',T(m))\\
                   &\leq& d(m,m')+\gamma^{-1}(|d(x,y)-d(T(x),T(y))|)\\
                   &\leq& d(m.m')+\gamma^{-1}(d(x,T(x))+d(y,T(y))).
\end{eqnarray*}
Choose $p$ as the midpoint of the segment $[m,m']$, applying the $2$-uniformly convexity, we have
\[
  \frac{k}{8}d^2(m,m') \leq \frac{1}{2}d^2(x,m)+\frac{1}{2}d^2(x,m')-d^2(x,p),
\]
similarly
\[
  \frac{k}{8}d^2(m,m') \leq \frac{1}{2}d^2(T(x),m)+\frac{1}{2}d^2(T(x),m')-d^2(T(x),p),
\]
\[
  \frac{k}{8}d^2(m,m') \leq \frac{1}{2}d^2(y,m)+\frac{1}{2}d^2(y,m')-d^2(y,p),
\]
\[
  \frac{k}{8}d^2(m,m') \leq \frac{1}{2}d^2(T(y),m)+\frac{1}{2}d^2(T(y),m')-d^2(T(y),p),
\]
Since $d^2(x,p)+d^2(y,p)\geq d^2(x,m)+d^2(y,m)$ and $d^2(T(x),p)+d^2(T(y),p)\geq d^2(T(x),m')+d^2(T(y),m')$,
we could obtain the following
\begin{eqnarray*}
  \frac{k}{2}d^2(m,m') &\leq& \frac{1}{2}d^2(x,m')-\frac{1}{2}d^2(T(x),m')+\frac{1}{2}d^2(T(x),m)-\frac{1}{2}d^2(x,m)\\
                      &&+\frac{1}{2}d^2(y,m')-\frac{1}{2}d^2(T(y),m')+\frac{1}{2}d^2(T(y),m)-\frac{1}{2}d^2(y,m)\\
                     &\leq& 2D(d(x,T(x))+d(y,T(y))
\end{eqnarray*}
where $D=diam(K)$. Thus
\begin{eqnarray*}
  J_T (m)&\leq& \sqrt{\frac{4D}{k}(d(x,T(x))+d(y,T(y))}+\gamma^{-1}(d(x,T(x))+d(y,T(y)))\\
         &\leq& \alpha(\max\{J_T (x), J_T (y)\}),
\end{eqnarray*}
where $\alpha(t)=\sqrt{\frac{8D}{k}t}+\gamma^{-1}(2t)$.

The third implication is immediate.
\end{proof}

We finish this section with the equivalent result of Theorem 3.14 in \cite{kp} and \cite{el} for $p$-uniformly convex spaces.
\begin{theorem}
Let $K$ be a bounded closed convex subset of $X$ a complete $p$-uniformly convex space, and let
$T : K \to X$ be continuous and of convex type. Suppose
\[
  \inf\{d(x, T(x)) : x \in K\} = 0.
\]
Then $T$ has a fixed point in $K$.
\end{theorem}
\begin{proof}
Let $x_0 \in X$ be fixed and define
\[
  \rho_0=\inf \{\rho>0: \inf \{d(x, T(x)) | x \in B(x_0,\rho) \cap K\}=0\}.
\]
Since $K$ is bounded, $\rho_0<\infty$. Moreover if $\rho_0=0$ then $x_0 \in K$ and $T(x_0)=x_0$ by the continuity of $T$. So assume that
$\rho_0>0$. Choose $(x_n) \subset K$ such that $d(x_n,T(x_n)) \to 0$ and $d(x_n,x_0) \to \rho_0$. It suffices to show that $(x_n)$ is convergent to
prove the theorem. If not, then there exists a $\varepsilon>0$ and subsequences $(u_k)$ and $(v_k)$ of $(x_n)$ such that $d(u_k,v_k)\geq \varepsilon$. Again, if necessary we may suppose $d(u_k,x_0)\leq \rho_0+\frac{1}{k}$ and $d(v_k,x_0)\leq \rho_0+\frac{1}{k}$. Denote $m_k$ as the midpoint of the segment $[u_k,v_k]$. Then applying the $p$-uniformly convexity to triangle $\triangle(x_0,u_k,v_k)$ we have
\begin{eqnarray*}
  d^p(x_0,m_k)&\leq& \frac{1}{2}d^p(x_0,u_k)+\frac{1}{2}d^p(x_0,v_k)-\frac{C}{8}d^p(u_k,v_k)\\
              &\leq& (\rho_0+\frac{1}{k})^p-\frac{C}{8}\varepsilon^p.
\end{eqnarray*}
We consider $k$ large enough, such that
\[
  d(x_0,m_k)\leq \bar{\rho}<\rho_0.
\]
On the other hand, since $T$ is of convex type, $\lim_{k\to \infty}d(m_k,T(m_k))=0$. This contradicts the definition of $\rho_0$.
\end{proof}


\bigskip

\begin{tabbing}

Renlong Miao\\

Institut f\"ur Mathematik, \\

Universit\"at Z\"urich,\\
 Winterthurer Strasse 190, \\

CH-8057 Z\"urich, Switzerland\\

{\tt renlong.miao@math.uzh.ch}\

\end{tabbing}


\begin{thebibliography}{ccccc}


\bibitem[BCL]{bcl}
Keith ~Ball, Eric ~A.~Carlen and Elliott~H.~Lieb {\it Sharp uniform convexity and smoothness inequalities for trace norms},
Invent math. 115, 463--482 (1994).

\bibitem[DKS]{dks}
S.~Dhompongsa, W.A.~Kirk and B.~Sims, {\it Fixed points of uniformly lipschitzian mappings}, Nonlinear Analysis 65 (2006) 762–772.

\bibitem[EF]{ef}
R.~Espinola and A.~Fernandez-Leon, {\it $\CAT(\kappa)$ spaces, weak convergence and fixed points},
Journal of Mathematical Analysis and Applications, vol. 353, no. 1, pp. 410--427, 2009.

\bibitem[GK]{gk}
K.~Goebel, W.~A.~Kirk, {\it Topics in Metric Fixed Point Theory}, 
Cambridge Univ. Press, Cambridge, 1990.

\bibitem[K]{kk}
K.~Kuwae, {\it Jensen’s inequality on convex spaces}, preprint (2010).

\bibitem[KP]{kp}
W.~A.~Kirk and B.~Panyanak, {\it A concept of convergence in geodesic spaces},
Nonlinear Anal.68 (12) (2008), 3689--3696.

\bibitem[L1]{ll}
E.~A.~Lifshic, {\it A fixed point theorem for operators in strongly convex spaces}.
Voronez. Gos. Univ. Trudy Mat. Fak. 16(1975), 23--28.

\bibitem[L2]{l2}
T.~C.~Lim, {\it Fixed point theorems for uniformly Lipschitzian mappings in $L^{p}$ spaces}.
Nolinear Analysis. Theory, Methods\& Application. Vol. 7, No. 5, pp. 555--563, 1983.

\bibitem[LS]{ls}
Urs ~Lang and Viktor ~Schroeder, {\it Jung's Theorem for Alexandrov Spaces of Curvature Bounded Above}'
Annals of Global Analysis and Geometry 15: 263--275, 1997.

\bibitem[LX]{th}
T.~C.~Lim and H.~K.~Xu, {\it Uniformly lipschitzian mappings in metric spaces with uniform normal structure},
Nonlinear Anal. 25(1995), 1231--1235.

\bibitem[O]{oo}
S.~I.~Ohta, {\it Convexities of metric spaces},
Geom. Dedicata 125, (2007), no. 1, 225--250.

\end{thebibliography}
\end{document}